\documentclass[12pt]{amsart}
\usepackage{amssymb, tikz}
\usepackage{amscd}
\usetikzlibrary{matrix,arrows}
\usepackage{tikz-cd}
\usetikzlibrary{matrix}
\usepackage[margin=3cm]{geometry}
\usepackage{scrextend}
\usepackage{hyperref}

\newtheorem{theorem}{Theorem}[section]
\newtheorem{lemma}[theorem]{Lemma}
\newtheorem{prop}[theorem]{Proposition}
\newtheorem{cor}[theorem]{Corollary}
\newtheorem{conj}[theorem]{Conjecture}

\theoremstyle{definition}
\newtheorem{defn}[theorem]{Definition}
\newtheorem{remark}[theorem]{Remark}
\newtheorem{example}[theorem]{Example}
\newtheorem{asm}[theorem]{Assumption}

\newcommand{\PP}{{\mathbb P}}
\newcommand{\Q}{{\mathbb Q}}
\newcommand{\R}{{\mathbb R}}
\newcommand{\Z}{{\mathbb Z}}
\newcommand{\Aut}{\mathrm{Aut}}	
\newcommand{\calM}{\mathcal M}
\newcommand{\calO}{\mathcal O}
\newcommand{\calI}{\mathcal I}
\newcommand{\calD}{\mathcal D}
\newcommand{\calX}{\mathcal X}
\newcommand{\calT}{\mathcal T}
\newcommand{\calL}{\mathcal L}
\newcommand{\calF}{\mathcal F}
\newcommand{\calU}{\mathcal U}
\newcommand{\Sing}{\mathrm{Sing}}
\newcommand{\Xv}{X^{\nu}}
\newcommand{\De}{D_{\epsilon}}
\newcommand{\Den}{D_{\epsilon,n}}

\newcommand{\ps}{\pi_{\textrm{ss}}}

\makeatletter
\newcommand{\xdashrightarrow}[2][]{\ext@arrow 0359\rightarrowfill@@{#1}{#2}}
\def\rightarrowfill@@{\arrowfill@@\relax\relbar\rightarrow}
\def\arrowfill@@#1#2#3#4{%
  $\m@th\thickmuskip0mu\medmuskip\thickmuskip\thinmuskip\thickmuskip
   \relax#4#1
   \xleaders\hbox{$#4#2$}\hfill
   #3$%
}\makeatother

\begin{document}
\title{A fibered power theorem for pairs of log general type} 
\author{Kenneth Ascher}
\email{kenascher@math.brown.edu}
\address{Mathematics Department, Brown University}
\author{Amos Turchet}
\email{tamos@chalmers.se}
\address{Mathematical Sciences, Chalmers University}
\thanks{Research of Ascher supported in part by funds from NSF grant DMS-1162367.}
\date{\today}
\begin{abstract} Let $f: (X,D) \to B$ be a stable family with log canonical general fiber. We prove that, after a birational modification of the base $\widetilde{B} \to B$, there is a morphism from a high fibered power of the family to a pair of log general type.  If in addition the general fiber is openly canonical, then there is a morphism from a high fibered power of the original family to a pair openly of log general type. \end{abstract}

\maketitle

\section{Introduction}
\noindent We work over an algebraically closed field of characteristic 0.\\

Analyzing the geometry of fibered powers of families of varieties has been pivotal in showing that well known conjectures in Diophantine geometry imply uniform boundedness of rational points on algebraic varieties of general type. In 1997, Caporaso, Harris, and Mazur, show in their celebrated paper \cite{chm}, that various versions of Lang's conjecture imply uniform boundedness of rational points on curves of general type. More precisely, they show that, assuming Lang's conjecture, for every number field $K$ and integer $g \geq 2$, there exists an integer $B(K,g)$ such that no smooth curve of genus $g \geq 2$ defined over $K$ has more than $B(K,g)$ rational points. Similar statements were proven for the case of surfaces of general type by Hassett \cite{has}, and eventually all positive dimensional varieties of general type in a series of two papers by Abramovich and  Abramovich-Voloch (\cite{dan} and  \cite{av}).  The essence of these papers, is a purely algebro-geometric statement: the proof of a ''fibered power theorem'', which analyzes the behavior of families of varieties of general type. 

The goal of our current research program is to apply birational geometry and moduli theory to show that the \emph{Vojta-Lang conjecture} (see Conjecture \ref{vl}) implies uniform boundedness of \emph{stably integral points} (see \cite{dan2}) on varieties of \emph{log general type}.   In short, our objective is to generalize the uniformity results mentioned above from varieties of general type to pairs of log general type. The first step is a generalization of the fibered power theorem alluded to above to the context of pairs, and we take the goal of this paper to be the proof of the following theorem:

\begin{theorem}\label{fiber} Let $(X,D) \to B$ be a stable family with integral and log canonical general fiber over a smooth projective variety $B$. Then after a birational modification of the base $\widetilde{B} \to B$, there exists an integer $n > 0$, a positive dimensional pair $(\widetilde{W},\widetilde{\Delta})$ of log general type, and a morphism $(\widetilde{X}^n_B, \widetilde{D}_n) \to (\widetilde{W}, \widetilde{\Delta})$.\end{theorem}

Moreover, with an additional assumption on the singularities of the general fiber, we can avoid a modification of the base:

\begin{theorem}\label{fiber2} Let $(X,D) \to B$ be a stable family with integral, openly canonical, and log canonical general fiber (see Definition \ref{open})  over a smooth projective variety $B$. Then there exists an integer $n > 0$, a positive dimensional pair $(W, \Delta)$ openly of log general type, and a morphism $(X^n_B, D_n) \to (W, \Delta)$. \end{theorem}

 We state two theorems to satisfy the two competing definitions of \emph{log general type}, one which often appears in birational geometry, and one which is useful for arithmetic applications. To remove any ambiguity, we refer to the former as a pair of log general type, and the latter as a pair openly of log general type. See Definitions \ref{def} and \ref{def2} for precise definitions.\\

Next, we wish to briefly outline how the previous theorem can be used for arithmetic applications. The analogue of Lang's conjecture in the setting of pairs is the following conjecture, due to Vojta and reformulated using ideas of Lang:

\begin{conj}\label{vl}[Vojta-Lang]:  Let $K$ be a number field and let $S$ be a finite set of places of $K$. Let $(X,D)$ be a pair openly of log general type defined over $K$. Then the set of $S$-integral points of $X \setminus D$ is not Zariski dense. \end{conj}

Given a stable family, Theorem \ref{fiber2} gives a morphism from a high fibered power of this family to a pair $(W, \Delta)$ openly of log general type. Conjecture \ref{vl} then predicts that the integral points of $W \setminus \Delta$ are not Zariski dense. This allows us to prove uniformity results for integral points on higher dimensional pairs of log general type assuming that the Vojta-Lang conjecture holds \cite{at}.  \\

In what follows, we give the two definitions used for pairs (openly) of log general type.

\begin{defn}\label{def} A pair $(X,D)$ of a proper variety $X$ and an effective $\Q$-divisor $D$ is of \textbf{log general type} if:
\begin{itemize}
\item $(X,D)$ has log canonical singularities and
\item $\omega_X(D)$ is big.
\end{itemize}\end{defn}

 For applications to arithmetic (for instance in the upcoming \cite{at}), it will be useful to consider the following.

\begin{defn}\label{def2} Let $X$ be a quasi-projective variety and let $\widetilde{X} \to X$ be a desingularization. Let $\widetilde{X} \subset Y$ by a projective embedding and suppose $D = Y \setminus \widetilde{X}$ is a divisor of normal crossings. Then $X$ is \textbf{openly of log general type} if $\omega_Y(D)$ is big. \end{defn}

Note that this second definition is independent of both the choice of the desingularization as well as the embedding; it is also a birational invariant. Moreover, definitions \ref{def} and \ref{def2} are equivalent if the pair $(X,D)$ has log canonical singularities and one considers the variety $X \setminus (D \cup \Sing(X))$.\\

Just to reiterate, we will refer to Definition \ref{def} by saying $(X,D)$ is a \emph{pair of log general type}. We will refer to Definition \ref{def2} by stating that the pair is \emph{openly of log general type}, as the definition is motivated by considering the complement $X \setminus D$. Throughout the course of this paper, we will take care to specify which definition we are using. 

\begin{defn}\label{open} By \emph{openly canonical}, we mean that the variety  $X \setminus D$ has canonical singularities. \end{defn}

For definitions of canonical, log canonical singularities, and stable pairs, see Definitions \ref{lc} and \ref{stable}.
\subsection{Ideas of proof}

To prove Theorem \ref{fiber}, we show that it suffices to prove the following:

\begin{theorem}[See Theorem \ref{thm3}] Let $(X,D) \to B$ be a stable family with integral and log canonical general fiber over a smooth projective variety $B$. Suppose that the variation of the family $f$ is maximal (see Definition \ref{max}). Let $G$ be a finite group such that $(X,D) \to B$ is $G$-equivariant. Then there exists an integer $n > 0$ such that the quotient of the pair by a finite group of automorphisms, $(X_B^n/G, D_n/G)$ is of log general type. \end{theorem}

Similarly, to prove Theorem \ref{fiber2}, it suffices to prove the following: \\

\begin{theorem}[See Theorem \ref{thm38}] Let $f: (X,D) \to B$ be a stable family with integral, openly canonical, and log canonical general fiber over a smooth projective variety $B$. Suppose that the variation of the family $f$ is maximal.  Let $G$ be a finite group such that $(X,D) \to B$ is $G$-equivariant.  Then there exists an integer $n > 0$ such that the quotient $(X^n_B/G, D_n/G)$ is openly of log general type. \end{theorem}

For a definition of maximal variation, see Definition \ref{max}. \\

We then obtain Theorem \ref{fiber2} by means of Theorem \ref{fiber} and Theorem \ref{thm3}. More specifically, we show that there is a birational transformation from $(\widetilde{W}, \widetilde{\Delta}) \to (W, \Delta)$, such that $(\widetilde{W}, \widetilde{\Delta})$ manifests $(W, \Delta)$ as a pair openly of log general type.\\

The main tool of this paper is a recent result of Kov\'acs-Patakfalvi which  says that  given a stable family with maximal variation $f: (X,\De) \to B$ where the general fiber is Kawamata log terminal (klt), then for large $m$ the sheaf $f_*(\omega_f(\De))^m$ is \emph{big} \cite[Theorem 7.1]{kp}. Here, the divisor $\De$ denotes the divisor with lowered coefficients $(1-\epsilon)D$ for a small rational number $\epsilon$. Unfortunately the result of \cite{kp} does \emph{not} hold for log canonical pairs (see Example 7.5 of \cite{kp}). As a result, since $D$ is not assumed to be $\Q$-Cartier, one obstacle of this paper is showing that bigness of $\omega_{X^n_B}(\Den)$ for some $n$ large enough allows you to conclude bigness of $\omega_{X^n_B}(D_n)$. To do so, one must first take a $\Q$-factorial dlt modification, followed by a relative log canonical model. The ideas here are present in Propositions 2.9 and 2.15 of \cite{px}. See Remark \ref{klt} for a more detailed discussion.

Finally, we must guarantee that the fibered powers are not too singular. A priori, it is unclear if taking high fibered powers to ensure the positivity of $\omega_{X^n_B}(D_n)$ leads to a pair with good singularities. This is ensured by the following statement:\\

\begin{prop}[See Proposition \ref{sing}] Let $f: (X,D) \to B$ be a stable family with integral and log canonical general fiber over a smooth projective variety $B$. Then for all $n >0$ the fibered powers $(X^n_B, D_n)$ have log canonical singularities. \end{prop}

This statement also works after taking quotients by finite groups of automorphisms:\\

\begin{prop} [See Corollary \ref{quotsing}] Let $f: (X,D) \to B$ be an slc family with integral and log canonical general fiber over a smooth projective variety $B$. Then for $n$ large enough, the quotient pair  $(X^n_B/G, D_n/G)$ also has log canonical singularities. \end{prop}

In fact, although we do not use it in this paper, we prove:\\

\begin{prop}[See Proposition \ref{totalspace}] The total space of the fiber product of stable families over a stable base is stable. \end{prop}

The main result we seek then follows via the above methods after applying standard tools from moduli theory.

\subsection{Previous work}

The general ideas present in this paper originated in the work of Caporaso, Harris, and Mazur, who showed in \cite{chm} that the fibered power theorem is true for families of curves of general type; i.e. families with general fiber smooth curves of genus $g \geq 2$. More precisely, they proved the following theorem: \begin{theorem}[{\cite[Theorem 1.3]{chm}}]\label{chm} Let $f: X \to B$ be a proper morphism of integral varieties, whose general fiber is a smooth curve of genus $g \geq 2$. Then for $n$ sufficiently large, $X^n_B$ admits a dominant rational map $h: X^n_B \to W$ to a positive dimensional variety of general type $W$. \end{theorem}
In this same paper, the authors conjectured that this result is actually true for families of varieties of general type of arbitrary dimension. However, the authors focused only on the case of curves as their results relied upon the nice description of a compact moduli space parameterizing mildly singular objects-- the moduli space of genus $g$ stable curves,  $\overline{\calM}_g$.  After Alexeev, Koll\'ar, and Shepherd-Barron constructed  a moduli space parametrizing stable surfaces analogous to that of stable curves, Hassett showed in \cite{has} that the correlation theorem was also true for families of surfaces of general type.  Abramovich later proved the fibered power theorem for varieties of general type of arbitrary dimension, using an analogue of semistable reduction \cite{dan}.

In \cite{AM}, Abramovich and Matsuki proved a fibered power theorem for principally polarized abelian varieties, using Alexeev's compact moduli space. This was a generalization of some prior results which considered pairs of elliptic curves and their origins (see \cite{pac2}, \cite{pac}, and \cite{dan2}). Furthermore, Abramovich and Matsuki remark that the result should in fact be true for stable pairs $(X,D)$, and so we take this as the goal of this paper. 

Finally, although we hope that this result will be of interest in its own right from the viewpoint of geometry, there are numerous applications to arithmetic, namely to the study of integral points on pairs of log general type, assuming the Lang-Vojta conjecture. This approach was taken in all the above mentioned papers and we obtain similar results in this direction in  \cite{at}. \\

\subsection*{Outline}
\begin{addmargin}[45pt]{0\linewidth}
\begin{itemize}
\item[Sec. 2 (Pg. \pageref{1})] Preliminary definitions and notation,
\item[Sec. 3 (Pg. \pageref{2})] We prove $\omega_{X_B^n}(D_n)$ is big for a stable family of maximal variation with log canonical general fiber. We prove that the fibered power theorem holds for max variaiton families when the general fiber is both openly canonical and log canonical.
\item[Sec 4 (Pg. \pageref{3})] Some results on singularities, namely we analyze the singularities of fibered powers and study the affect of group quotients. We prove the fibered power theorem for log canonical general fiber in the case of max variation. 
\item[Sec 5 (Pg. \pageref{4})] We prove the full fibered power theorems by reducing to families of maximal variation.

\end{itemize}
\end{addmargin}

\subsection*{Acknowledgements} We thank Brendan Hassett for suggesting this problem. We thank Dan Abramovich for his constant support and many useful conversations. K. Ascher thanks Dori Bejleri and Patricio Gallardo for many helpful conversations, as well as S\'andor Kov\'acs and Zsolt Patakfalvi for discussing their results with him at length. We also thank them for providing us with an early draft of \cite{kp}. K. Ascher thanks the Mathematics department at Chalmers University for providing an excellent working environment.  Finally, we thank the referee for their helpful suggestions which improved the exposition of this paper.

\section{Preliminaries and Notation}\label{1}

\subsection{Birational Geometry}
\begin{defn}
A line bundle $\calL$ on a proper variety $X$ is called \textbf{big} if the global sections of $\calL^m$ define a birational map for $m > 0$. A Cartier divisor $D$ is called big if $\calO_X(D)$ is big.  
\end{defn}
From the point of view of birational geometry and the minimal model program, it has become convenient and standard to work with pairs. We define a pair $(X,D)$ to be a variety $X$ along with an effective $\R$-divisor $D = \sum d_i D_i$ which is a linear combination of distinct prime divisors. 

\begin{defn}\label{lc}
Let $(X,D)$ be a pair where $X$ is a normal variety and $K_X + D$ is $\Q$-Cartier. Suppose that there is a log resolution $f: Y \to X$ such that $$K_Y + \sum a_E E = f^*(K_X + D),$$ where the sum goes over all irreducible divisors on $Y$. We say that $(X,D)$ is:
\begin{itemize} 
\item \emph{canonical} if all $a_E \leq 0$, 
\item \emph{log canonical} (lc) if all $a_E \leq 1$, and
\item \emph{Kawamata log terminal} (klt) if all $a_E < 1$.
\end{itemize}
\end{defn}

\begin{remark} In particular, for a klt pair, the coefficients $d_i$ in the decomposition $D = \sum d_i D_i$ are all \emph{strictly} $< 1$.  Similarly, for a lc pair, the coefficients are $\leq$ 1. \end{remark}

\begin{defn} A  pair $(X,D = \sum d_i D_i)$ is \textbf{semi-log canonical} (slc) if $X$ is reduced,  $K_X + D$ is $\Q$-Cartier and:
\begin{enumerate} 
\item The variety $X$ satisfies Serre's condition S2,
\item  $X$ is Gorenstein in codimension 1, and 
\item if $\nu: \Xv \to X$ is the normalization, then the pair $(\Xv, \sum d_i \nu^{-1}(D_i) + D^{\vee})$ is log canonical, where $D^{\vee}$ denotes the preimage of the double locus on $\Xv$.
\end{enumerate}
\end{defn}

\begin{remark} Semi-log canonical singularities can be thought of as the extension of log canonical singularities to non-normal varieties. The only difference is that a log resoultion is replaced by a \emph{good semi-resolution}.  \end {remark}

\begin{defn}\label{stable} A pair $(X,D)$ of a proper variety $X$ and an effective $\Q$-divisor $D$, is a \textbf{stable pair} if:
\begin{itemize}
\item The line bundle $\omega_X(D)$ is ample and 
\item The pair $(X,D)$ has semi-log canonical singularities
\end{itemize}
\end{defn}

We assume that all of our families of pairs satisfy \emph{Koll\'ar's condition}. To be precise, we define a stable family as follows. Let $X$ be a variety and $\calF$ an $\calO_X$-module. The dual of $\calF$ is denoted $\calF^{\star} := \mathcal Hom_X (\calF, \calO_X)$ and we define the $m$-th reflexive power of $\calF$ to be the double dual (or reflexive hull) of the $m$-th tensor power of $\calF$: $$\calF^{[m]} := (\calF^{\otimes m})^{\star \star}.$$

\begin{defn} An \textbf{slc family} is a flat morphism $f: (X,D) \to B$ such that:
\begin{itemize}
\item each fiber $(X_b, D_b)$ is an slc pair,
\item $\omega_f(D)^{[m]}$ is flat, and
\item (Koll\'ar's Condition) for every base change $\tau: B' \to B$, given the induced morphism $\rho: (X_{B'}, D_{B'}) \to (X,D)$ we have that the natural homomorphism $$\rho^*(\omega_f(D)^{[m]}) \to \omega_{f'}(D)^{[m]}$$ is an isomorphism
\end{itemize}

\noindent We say that $f: (X,D) \to B$ is a \textbf{stable family} if in addition to the above, each $(X_b, D_b)$ is a stable pair. Equivalently, $K_{X_b} + D_b$ is ample for every $b \in B$. \end{defn}

\subsection{Moduli space of stable pairs}
Constructing the moduli space  of stable pairs, denoted below by $\overline{M}_h$, has been a difficult task. A discussion of the construction of the moduli space $\overline{M}_h$ is not necessary for this paper, but for sake of completeness we note that there exists a finite set of constants, which we denote by $h$, that allows for a compact moduli space. As long as the coefficients $d_i$ appearing in the divisor decomposition are all $ > \frac{1}{2}$, there are no issues and we do in fact have a well defined moduli space. There is no harm in assuming this outright, which leads to the following:

\begin{asm} The coefficients of the divisors considered are always $> \frac{1}{2}$.\end{asm}

\noindent We refer the reader to \cite{kol} or to the introduction of \cite{kp} for more details.

\subsection{Variation of Moduli}
Given a stable family $f:(X,D) \to B$, we obtain a \textit{canonical morphism}:
$$\phi: B \to \overline{M}_h$$ 
sending a point $b \in B$ to the point of the moduli space $\overline{M}_h$ of stable pairs, classifying the fiber $(X_b, D_b)$. This motivates the following definition.

\begin{defn}\label{max} A family has \textbf{maximal variation} of moduli if the corresponding canonical morphism is generically finite. \end{defn}

Equivalently, the above definition means that the family is a truly varying family, diametrically opposed to one which is \emph{isotrivial}, where the fibers do not vary at all.

\subsection{Notation} Given a morphism of pairs $f: (X,D) \to B$, we denote by $(X_B^n,D_n)$ the unique irreducible component of the $n$th fiber product of $(X,D)$ over $B$ dominating $B$. We define  $D_n$ to be the divisor $D_n := \sum_{i=1}^n \pi_i^*(D)$ where the maps $\pi_i: (X^n_B, D_n) \to B$ denote the projections onto the $i$th factors. We denote by $f_n$ the maps $f_n: (X^n_B, D_n) \to B$. Finally, we denote by $\De$ the divisor $(1-\epsilon)D$ and by $\Den$ the sum $\Den := \sum_{i=1}^n \pi_i^*(\De)$. \\

\section{Positivity of the relative anti-canonical sheaf} \label{2}
Recall that to prove that the pair $(X^n_B, D_n)$ is a \emph{pair of log general type}, we must show that

\begin{enumerate}
\item[(a)]  $\omega_{X^n_B}(D_n)$ is big and 
\item[(b)] The pair $(X^n_B, D_n)$ has log canonical singularities. 
\end{enumerate}

We also remind the reader that we will demonstrate in Section 4 that Theorem \ref{thm3} implies Theorem \ref{fiber}. Therefore, in this section we will assume that the variation of our family is maximal. More precisely, the goal of this section is to prove the following proposition, tackling (a) of the above definition:

\begin{prop}\label{lcbig} Let $f:(X,D) \to B$ be a stable family with maximal variation over a smooth, projective variety $B$ with integral and log canonical general fiber, then for $n$ sufficiently large, the sheaf $\omega_{X^n_B}(D_n)$ is big. \end{prop}

As mentioned in the introduction, we will prove the above proposition by means of the following slightly weaker statement:

\begin{prop}\label{kltbig} Let $f:(X,\De) \to B$ be a stable family with maximal variation over a smooth, projective variety $B$ with klt general fiber, then for $n$ sufficiently large, the sheaf $\omega_{X^n_B}(\Den)$ is big. \end{prop}

Our proof of Proposition \ref{kltbig} requires the recent Theorem of Kov\'acs and Patakfalvi:

\begin{theorem}[{\cite[Theorem 7.1, Corollary 7.3]{kp}}]\label{kp} If $f: (X,\De) \to B$ is a stable family with maximal variation over a normal, projective variety $B$ with klt general fiber, then $f_*(\omega_f(\De)^m)$ is big for $m$ large enough. Moreover, $\omega_f(\De)$ is big. \end{theorem}

Let $S^{[n]}$ denote the reflexive hull of the $n$th symmetric power of a sheaf. Then the above theorem is equivalent to saying that, under the hypotheses, for any ample line bundle $H$ on $B$ there exists an integer $n_0$ such that  \begin{equation}S^{[n_0]}(f_*(\omega_f(\De)^m) \otimes H^{-1}\end{equation} is generically globally generated. We desire to show that this implies Proposition \ref{kltbig}.\\

The proof of this statement essentially follows from Proposition 5.1 of \cite{has}, but we include the proof for completeness to show how it extends to the case of pairs.  We begin with a lemma:

\begin{lemma}\label{gorenstein} Let $f: (X,D) \to B$ be a stable family over a smooth projective variety $B$ such that the general fiber has log canonical singularities. Then for all $n > 0$, the following formula holds: $$\omega_{X^n_B}(D_n)^{[m]} = \pi_1^* \omega_f(D)^{[m]} \otimes \cdots \otimes \pi_n^* \omega_f(D)^{[m]} \otimes f_n^* \omega_B^m.$$ \end{lemma}

\begin{proof} Recall that the relative dualizing sheaf satisfies the following equation: $$\omega_{f_n}(D_n) = \pi_1^* \omega_f(D) \otimes \cdots \otimes \pi_n^* \omega_f(D)$$ where $\pi_j$ denotes the projection $\pi_j: X^n \to X$ to the $j$th factor. Since $B$ is smooth we obtain:  $$ \omega_{X^n_B}(D_n)^{[m]} = \omega_{f_n}(D_n)^{[m]} \otimes f_n^* \omega_B^m.$$ Since $f: (X,D) \to B$ is a stable family, there exists an integer $m$ such that for all $b \in B$, the sheaf $\omega_{f}(D)^{[m]}|_{X_b}$ is locally free. Moreover, since this sheaf is locally free on each fiber, $\omega_f(D)^{[m]}$ is also locally free for this $m$. We claim that the following holds:$$ \omega_{X^n_B}(D_n)^{[m]} = \pi^*_1{\omega_f(D)}^{[m]} \otimes \cdots \otimes \pi^*_n{\omega_f(D)}^{[m]} \otimes f_n^* \omega_B^m.$$ Both sides of the equation are reflexive -- the left hand side by construction, and the right hand side because it is the tensor product of locally free sheaves. Therefore, to prove the equivalence, we must show the two sides agree on an open set whose complement has codimension at least two. Consider the locus consisting of both the general fibers, which are log canonical and hence $\Q$-Gorenstein, as well as  the nonsingular parts of the special fibers. Note that the complement of this locus is of codimension at least two, because the singular parts of the special fiber are of codimension one, thus of at least codimension two in the total space.

\end{proof}

We will now give a proof of Proposition \ref{kltbig}. 

\begin{proof}[Proof of Proposition \ref{kltbig}] Let $m \in \Z$ be such that both $\omega_{f}(\De)^{[m]}$ is locally free and $f_*(\omega_f(\De)^m)$ is big. First note that for $n$ large enough, the sheaf $$\big(f_*(\omega_f(\De)^m)\big)^{[n]} \otimes H^{-1}$$ is generically globally generated. This follows since by Proposition 5.2 of \cite{has}, for a $r$-dimensional vector space $V$, each irreducible component of the reflexive hull of the $m$th tensor power of $V$,  is a quotient of a representation $S^{[q_1]}(V) \otimes \cdots \otimes S^{[q_k]}(V)$, where $k = r!$. Using this, we will prove that $\omega_{X^n_B}(\Den)$ is big for large $n$. To do so, it suffices to show that there are on the order of $m^{n\dim X_{\eta}+b}$ sections of $\omega_{X^n_B}(\De)^{[m]}$ where $b = \dim B$, and $X_{\eta}$ denotes the general fiber. 

By Lemma \ref{gorenstein}, $$\omega_{X^n_B}(\Den)^{[m]} = \pi^*_1{\omega_f(\De)}^{[m]} \otimes \cdots \otimes \pi^*_n{\omega_f(\De)}^{[m]} \otimes f_n^* \omega_B^m.$$ The sheaf $\omega_f(\De)$ has good positivity properties-- it is big by Corollary 7.3 of \cite{kp}, but the sheaf $\omega_B$ is somewhat arbitrary and could easily prevent $\omega_{X^n_B}(\Den)$ from being big. However, taking high enough powers of $X$ allows the positivity of $\omega_{f}(\De)$ to overcome the possible negativity of $\omega_B$. 

Applying $(f_n)_*$ gives, via the projection formula:$$ (f_n)_*(\omega_{X^n_B}(\Den)^{[m]}) =  \big(f_* (\omega_f(\De)^{[m]})\big)^n \otimes \omega_B^m  $$  which is also a reflexive sheaf by Corollary 1.7 of \cite{hart}. More specifically, it is the push forward of a reflexive sheaf under a proper dominant morphism. Then the inclusion map $\omega_f(\De)^m \to \omega_f(\De)^{[m]}$ induces a map of reflexive sheaves:$$\big(f_* \omega_f(\De)^m\big)^{[n]} \to \big(f_* \omega_f(\De)^{[m]}\big)^n = (f_n)_*( \omega_{X^n_B}(\Den)^{[m]}) \otimes \omega_{B}^{-m}$$

which is an isomorphism at the generic point of $B$.\\

Let $H$ be an invertible sheaf on $B$ such that $H \otimes \omega_B$ is very ample. Then we can choose $n$ so that $\big(f_*\omega_f(\De)^m\big)^{[n]} \otimes H^{-m}$ is generically globally generated for all admissible values of $m$. But then$$\big(f_*\omega_f(\De)^m\big)^{[n]} \otimes H^{-m} = (f_n)_*(\omega_{X^n_B}(\Den)^{[m]}) \otimes (H \otimes \omega_B)^{-m}$$ 

is also generically globally generated for the same $m$.

This sheaf has rank on the order of $m^{n\dim X_{\eta}}$ so there are at least this many global sections. By our assumption on $H$, we have that $(H \otimes \omega_B)^m$ has on the order of $m^b$ sections varying horizontally along the base $B$.  By tensoring, we obtain that the sheaf $$(f_n)_* \left(\omega_{X^n_B}(\Den)\right)^{[m]}$$  has on the order of $m^{n\dim X_{\eta}+b}$ global sections, and therefore $\omega_{X^n_B}(\Den)$ is big.\end{proof}

\begin{remark}\label{klt} The above proposition assumed that the general fiber $(X_b, D_b)$ had klt singularities, but to prove Theorem 1.1 as stated, we must allow the general fiber to have log canonical singularities. Unfortunately, we cannot just raise the coefficients of $D$ so that the pair has log canonial singularities, via twisting by $\epsilon D$ to conclude that $\omega_{X^n_B}(D_n)$ is also big. This is because we do \emph{not} know that the divisor $D$ is $\Q$-Cartier. We remedy this situation with a $\Q$-factorial divisorial log terminal (dlt) modification (see Section 1.4 of \cite{sing} for an overview of dlt models), as explained below. \end{remark} 

First, the definition of a dlt pair.

\begin{defn} Let $(X,D)$ be a log canonical pair such that $X$ is normal an $D = \sum d_i D_i$ is the sum of distinct prime divisors. Then $(X,D)$ is \textbf{divisorial log terminal} (dlt) if there exists a closed subset $Z \subset X$ such that: \begin{enumerate} \item $X \setminus Z$ is smooth and $D|_{X \setminus Z}$ is a snc divisor \item If $f: Y \to X$ is birational and $E \subset Y$ is an irreducible divisor such that $\textrm{center}_X E \subset Z$ then the discrepancy $a(E, X, D) < 1$ \end{enumerate}\end{defn}

See Definition 2.25 in \cite{km} for a definition of the discrepancy of a divisor $E$ with respect to a pair $(X,D)$.

Roughly speaking, a pair $(X,D)$ is dlt if it is log canonical, and it is simple normal crossings at the places where it is not klt. The following theorem of Hacon guarantees the existence of dlt modifications.

\begin{theorem}{\cite[Theorem 3.1]{kk}}\label{dlt} Let $(X, D)$ be a pair of a projective variety and a divisor $D = \sum d_i D_i$ with coefficients $0 \leq d_i \leq 1$, such that $K_X + D$  is $\Q$-Cartier. Then $(X,D)$ admits a $\Q$-factorial minimal dlt model $f^{min}: (X^{min}, D^{min}) \to (X,D)$. \end{theorem}

The upshot here is that, starting with a log canonical pair $(X,D)$ we can obtain a model which is dlt and $\Q$-factorial. \\

The statement that we will actually apply follows from Proposition 2.9 of \cite{px}. 

\begin{prop}{\cite[Proposition 2.9]{px}}\label{qfact} Let $f: (X,D) \to B$ be a stable family over a smooth variety $B$. Assume that the general fiber $(X_b,D_b)$ has log canonical singularities and that the variation of the family is maximal. Then for each $0 < \epsilon \ll 1$ there exists a pair $(Z, \Delta_{\epsilon})$, an effective divisor $\Delta$ on $Z$, and a morphism $p: Z \to X$ such that:

\begin{itemize}
\item[(a)] $K_Z + \Delta = p^*(K_X + D)$,
\item[(b)] $(Z, \Delta_{\epsilon})$ is klt
\item[(c)] $g: (Z, \Delta_{\epsilon}) \to B$ is a stable family
\item[(d)] The variation of $g$ is maximal
\item[(e)] $\Delta - \Delta_{\epsilon}$ is an effective divisor such that $\textrm{Supp}(\epsilon \Delta) \subset \textrm{Ex}(p) \cap \textrm{Supp}( p_*^{-1} \Delta)$
\end{itemize}
\end{prop}
\begin{proof}[Sketch of proof]

The rough idea is to take a $\Q$-factorial dlt modification of $X$, and then shrink the resulting divisor so that the new pair $(\widetilde{Z}, \widetilde{\Delta}_{\epsilon})$ is klt. Finally, taking the relative log canonical model of $(\widetilde{Z}, \widetilde{\Delta}_{\epsilon}) \to X$ yields a stable family with klt general fiber and maximal variation. \end{proof}

Using the above discussion, we are now in position to prove the main statement of this section, Proposition \ref{lcbig}, whose proof is inspired by Proposition 2.15 of \cite{px}.

\begin{proof}[Proof of Proposition \ref{lcbig}] 
We begin with a stable family with maximal variation $f: (X,D) \to B$ such that the generic fiber is log canonical. The goal is to show that $\omega_{X^n_B}(D_n)$ is big for $n$ sufficiently large. 

First take $\widetilde{p}: \widetilde{Z} \to X$ to be a $\Q$-factorial dlt modification of $X$, and let $\widetilde{\Delta}$ be a divisor on $\widetilde{Z}$ such that $\widetilde{p}^*(K_X + D) = K_{\widetilde{Z}} + \widetilde{\Delta}$. Since $\widetilde{Z}$ is $\Q$-factorial by construction, we can lower the coefficients of the divisor $\widetilde{\Delta}$ by $0< \epsilon \ll 1$, a rational number, to obtain a klt pair $(\widetilde{Z}, \widetilde{\Delta}_{\epsilon})$. 

Define $p: Z \to X$ to be the relative log canonical model of $(\widetilde{Z}, \widetilde{\Delta}_{\epsilon}) \to X$. Denoting by $q: \widetilde{Z} \dashrightarrow Z$ the induced morphism, we define $\Delta$ to be the pushforward $\Delta = q_* (\widetilde{\Delta})$. 
By Proposition \ref{qfact}, the new family $g: (Z, \Delta_{\epsilon}) \to B$ is a stable family with maximal variation such that the generic fiber is klt. Thus, by Proposition \ref{kltbig}, for $n$ large enough, $\omega_{Z^n_B}(\Delta_{\epsilon,n})$ is big. 

Furthermore, since $(Z, \Delta_{\epsilon}) \to X$ is the relative log canonical model of $(\widetilde{Z}, \widetilde{\Delta}_{\epsilon}) \to X$, pluri-log canonical forms on $\widetilde{Z}$ are the pull back of pluri-log canonical forms on $Z$.  From this we conclude that $\omega_{\widetilde{Z}^n_B}(\widetilde{\Delta}_{\epsilon,n})$ is also big. Now since $\widetilde{Z}$ is $\Q$-factorial, we know that $\epsilon \Delta$ is a $\Q$-Cartier divisor. This property allows us to enlarge the coefficients of $\widetilde{\Delta}$. Recall that $\epsilon\Delta$ is effective by Proposition \ref{qfact} (e), and thus $\omega_{\widetilde{Z}^n_B}(\widetilde{\Delta}_n)$ is big as well.

Since $\widetilde{p}: \widetilde{Z} \to X$ is a birational morphism and $\widetilde{p}^*(K_X + D) = K_{\widetilde{Z}} + \widetilde{\Delta}$, pulling back pluri-log canonical forms through $\widetilde{p}$ preserves the number of sections. Thus, we finally conclude that $\omega_{X^n_B}(D_n)$ is big. \end{proof}

Finally, we prove the following theorem for pairs openly of log general type:

\begin{theorem}\label{openbig} Let $f: (X,D) \to B$ be a stable family with integral, openly canonical, and log canonical general fiber over a smooth projective variety $B$. Suppose that the variation of the family $f$ is maximal. Then $\omega_{X_{ss}^n}(D^{ss}_n)$ is big, where $(X_{ss}^n, D^{ss}_n)$ denotes the $n$th fibered power of the weak semistable model of the pair $(X,D)$. \end{theorem}

\begin{proof}
Consider the following diagram:
\begin{center}

\begin{tikzpicture}
\matrix(m)[matrix of math nodes,
row sep=2.6em, column sep=2.8em,
text height=1.5ex, text depth=0.25ex]
{(X_{ss}, D^{ss}) & (\widetilde{X}, \widetilde{D}) = X \times_{B_1} B & (X,D)\\
\Delta \subset B_1 & B_1 & B\\};
\path[->,font=\scriptsize,>=angle 90]
(m-1-1) edge node[auto] {$\phi$} (m-1-2)
        edge node[auto] {$\pi_{ss}$} (m-2-1)
(m-1-2) edge node[auto] {$\sigma$} (m-1-3)
	   edge node[auto] {$g$} (m-2-2)
(m-1-3) edge node[auto] {$f$} (m-2-3)
(m-2-1) edge node[auto] {} (m-2-2)
(m-2-2) edge node[auto] {} (m-2-3);
\end{tikzpicture}

\end{center}where $(X_{ss}, D^{ss}) \to B_1$ denotes the weak semistable model (see Definition 0.4 of \cite{karu}) of the family $(X,D) \to B$, and $\Delta \subset B_1$ denotes the discriminant divisor over which the exceptional lies. Such a model exists by \cite{karu}. Since taking the weak semistable model gives a pair which is at worst openly canonical and log canonical, we are not required to take a resolution of singularities. This is because, by definition of both openly canonical and log canonical singularities, sections of $\omega_{X_{ss}}(D^{ss})$ give regular sections of logarithmic pluricanonical sheaves of any desingularization.

More precisely, we have: $$\phi^{*} (\omega_{g}(\widetilde{D})) = \omega_{\ps}(D^{ss} + E) \subset \omega_{\ps}(D^{ss}+ \ps^{*}(\Delta)).$$

Let $\ps^*\Delta = \Delta^{ss}$.  Then since $\omega_{g}(\widetilde{D})$ is big by Theorem \ref{kp}, so is $\omega_{g}(\widetilde{D} - \frac{1}{n}\Delta^{ss})$. Taking fibered powers as in Proposition 2.2 shows that $\omega_{\widetilde{X}^n_{B_1}}(\widetilde{D}_n(-\Delta^{ss}_n))$ is also big. Moreover, $$\phi_n^*(\omega_{\widetilde{X}^n_{B_1}}(\widetilde{D}_n(-\Delta^{ss}_n))) \subset \omega_{X_{ss}^n}(D^{ss}_{n})(\Delta^{ss}_n - \Delta^{ss}_n) = \omega_{X_{ss}^n}(D^{ss}_{n})$$ is big. 
\end{proof}

The definition of openly of log general type then implies that we have actually shown the following:

\begin{theorem}\label{openthm} Let $f: (X,D) \to B$ be a stable family with integral, openly canonical, and log canonical general fiber over a smooth projective variety $B$. Suppose that the variation of the family $f$ is maximal. Then there exists an integer $n >0$ such that $(X^n_B, D_n)$ is openly of log general type. \end{theorem}

\section{Singularities} \label{3}
The purpose of this section is to prove that, assuming we begin with a pair $(X,D)$ with log canonical singularities, then fibered powers $(X^n_B, D_n)$ also have log canonical singularities for all $n > 0$. As the following example shows, it \emph{is} necessary to restrict the singularities, as there exist varieties $Y$ such that $\omega_Y$ is a big, but $Y$ is \emph{not} of general type! 

\begin{example} Let $Y$ be the projective cone over a quintic plane curve $C$. Then $\omega_Y$ is big (even ample), but $Y$ is birational to $\PP^1 \times C$, which has Kodaira dimension $\kappa(\PP^1 \times C) = -\infty$. So although $\omega_Y$ is big, $Y$ is \emph{not} openly of log general type. \end{example}

The following proposition is a version of log inversion of adjunction:

\begin{prop}[Lemma 2.12 \cite{zsolt}]\label{slcbase} The total space of an slc family over an slc base has slc singularities. \end{prop}

This immediately implies: 

\begin{cor}\label{fam} The total space of the product of slc families over an slc base also has slc singularities. \end{cor}

\begin{proof} Let $f: (X_1,D_1) \to B$ and $g: (X_2, D_2) \to B$ be two slc families over an slc base $B$. Then the product family $g: (X, \Delta) \to B$ is the total space of an slc family over either of the factors. Therefore both the product family as well as its total space have slc singularities  by Proposition \ref{slcbase}. \end{proof}

Inductively, this shows that the fibered powers $(X_B^n, D_n)$ have semi-log canonical singularities. The statement that we will actually use to prove our result is the following:

\begin{prop}\label{sing} Let $f:(X,D) \to B$ be a stable family with integral and log canonical general fiber over a smooth projective variety $B$. Then for all $n > 0$ the fibered powers $(X^n_B, D_n)$ have log canonical singularities. \end{prop}

\begin{proof} By Proposition \ref{slcbase}, the total space of the family $(X,D)$ is slc. In fact, we will show that it is actually log canonical, which is equivalent to showing that $(X,D)$ is normal. Recall to show that the pair $(X,D)$ is normal, it suffices to show that it is regular in codimension one (R1) and satisfies Serre's condition S2.  Since the general fiber has log canonical singularities, the fibers $(X_b,D_b)$ are R1 over the general point of the base $B$. Over the special fibers, the singularities are of at least codimension one in the fiber, and are thus at least codimension two in the total space. Therefore, it follows that the total space $(X,D)$ is R1. Finally, the pair $(X,D)$ is S2 by definition, since it has semi-log canonical singularities.

Therefore, by Corollary \ref{fam}, for all $n >0$ the fibered powers $(X^n_B, D_n)$ also have log canonical singularities.  

\end{proof}

In fact, the following stronger statements are also true. Although we do not use them in this paper, we hope that they may be of interest to readers:

\begin{prop} The fiber product of two stable families is a stable family. \end{prop}
\begin{proof} This result essentially follows from Proposition 2.12 in \cite{products}. We reproduce the argument for the convenience of the reader.

Let $f: (X,D) \to B$ and $g : (Y, E) \to B$ be two stable families, and denote the fiber product family by $h: (Z, F) \to B$. Since both families $f$ and $g$ are flat of finite type with $S_2$ fibers by assumption, and since we are assuming Koll\'ar's condition, by Proposition 5.1.4 of \cite{ah}, we have that $\omega_{X/B}^{[k]}(D)$ is flat over $B$. Moreover, by Lemma 2.11 of \cite{products},  we have that $$p^*_X \omega^{[k]}_{X/B}(D) \otimes p^*_Y \omega^{[k]}_{Y/B}(E)$$ is a reflexive sheaf on the product. By Lemma 2.6 of \cite{hk}, the above sheaf is isomorphic to $\omega^{[k]}_{Z/B}(F)$ on an open subset whose complement has codimension at least two, and therefore we  conclude that $$\omega^{[k]}_{Z/B}(F) = p^*_X \omega^{[k]}_{X/B}(D) \otimes p^*_Y \omega^{[k]}_{Y/B}(E).$$ Moreover, Koll\'ar's condition holds, as by assumption both components of this fiber product commute with arbitrary base change. Choosing a sufficient index $k$, namely the least common multiple of the index of the factors, we see that $\omega_{Z/B}(F)$ is a relatively ample $\Q$-line bundle and thus we conclude that $h: (Z,F) \to B$ is also a stable family.
\end{proof}

\begin{prop} \label{totalspace} The total space of the fiber product of stable families over a stable base is stable. \end{prop}
\begin{proof} By Proposition 2.15 in \cite{px} (see also \cite{fujino} Theorem 1.13), if $f: (X,D) \to (B,E)$ is a stable family whose variation is maximal over a normal base, then $\omega_f(D)$ is nef.  First we note that it suffices to prove the statement over a normal base, since nef is a property which is decided on curves. Since normalization is a finite birational morphism, nonnegative intersection with a curve is preserved. Thus, we wish to show that this statement is true without the assumption that the variation of $f$ is maximal.  Let $B' \to B$ be a finite cover of the base so that the pullback family $f': (X', D') \to B'$ maps to $g: (\calU, \calD) \to T$, a family of maximal variation. In this case $\omega_{f'}(D')$ is nef, as it is the pullback of $\omega_g(\calD)$ which is nef. Since $\omega_{f'}(D')$ is the pullback of $\omega_f(D)$ by the finite morphism $X' \to X$, the projection formula implies that $\omega_f(D)$ is nef as well. This shows that the sheaf $\omega_f(D)$ is nef, regardless of whether the variation of $f$ is maximal or not. Then since $\omega_f(D)$ is nef and $f$-ample, and since the base is stable, $\omega_B(E)$ is ample. Therefore, we can conclude that $\omega_X(D + f^*E)   = \omega_f(D) \otimes f^*\omega_B(E)$ is ample. \end{proof}

The following theorem that we actually need follows from Proposition \ref{lcbig} and Proposition \ref{sing}.

\begin{theorem}\label{thm2} Let $(X,D) \to B$ be a stable family with integral and log canonical general fiber and maximal variation over a smooth projective variety $B$. Then there exists an integer $n > 0$ such that the pair $(X_B^n, D_n)$ is of log general type. \end{theorem}

\begin{proof} By Proposition \ref{lcbig}, we have that $\omega_{X_B^n}(D_n)$ is big, and by Proposition \ref{sing}, the fibered powers $(X_B^n, D_n)$ have log canonical singularities. \end{proof}

To prove the stronger Theorem \ref{thm3}, we must show that what we have proven also works after taking the quotient by a group of automorphisms. This is precisely the content of Proposition \ref{group} and Corollary \ref{group2} below. 

This claim essentially follows from the work of various authors in previous papers in the subject. The approach is present in, for example, Lemma 3.2.4 of \cite{AM} as well as Lemma 2.4 of \cite{pac}. We reproduce the statement in our case below:

\begin{prop}\label{group} Let $(X,D)$ be openly of log general type. There exists a positive integer $n$ such that the pair $(X^n_B, D_n) / G$ is also openly of log general type. \end{prop}

\begin{proof}
Let $H \subset X$ be the locus of fixed points of the action of $G \subseteq \Aut(X,D)$. Let $\calI_H$ denote the corresponding sheaf of ideals. We have seen before that $\omega_f(D)$ is big. Then, for sufficiently large $k$, we have that the sheaf $$\omega_{f}(D)^{\otimes k} \otimes f^* \omega_B^{\otimes k} \otimes  \calI_H^{|G|}$$ is big.  If we pass to the $k$th fibered power, we have that $$(\omega_{X^k_B}(D_k))^{\otimes k} \otimes f_k^* \omega_B^{\otimes k} \otimes \Pi_{i=1}^k \pi^{-1}\calI_H^{|G|}$$ is also big. 

The product $\Pi_{i=1}^k \pi_i^{-1} \calI_H^{|G|} \subset \left(\displaystyle\sum_{i=1}^k \pi^{-1}_i \calI_H^{|G|}\right)^k$, and the latter ideal vanishes to order at least $k|G|$ on the fixed points of the action of $G$. Moreover, we have that $$(\omega_{f_k}(D_k))^{\otimes k} \otimes \pi^*_k \omega_B^{\otimes k} = (\omega_{X^k_B}(D_k))^{\otimes k}.$$ This allows us to conclude that for $n \gg 0$, there are enough invariant sections of $\omega_{X^k_B}(D_k)^{\otimes n}$ vanishing on the fixed point locus to order at least $n|G|$. 

Now let $$r: (\calX, \calD) \to (X^k_B, D_k)$$ be an equivariant good resolution of singularities so that $r^{-1}(D_k) = \calD$. Note that such a resolution is guaranteed by Hironaka \cite{H}. Since $X \setminus D$ does not necessarily have canonical singularities away from the general fiber, we have introduced exceptional divisors in the resolution that will alter sections of $\omega_X(D)$. To fix this, we simply apply the methods used in the proof of Theorem \ref{openbig}-- namely twist by some small negative multiple of the divisor $\Delta$ containing the exceptional. 

 To conclude the result, it suffices to show that invariant sections of $(\omega_{X^k_B}(D_k))^{\otimes n}$ vanishing on the fixed point locus to order at least $n\cdot|G|$, descend to sections of the pluri-log canonical divisors of a good resolution of the quotient pair $(X^n_B/G, D_n/G)$. \\

Denote by $q: (\calX, \calD) \to (\calX/G, \calD/G)$ the morphism to the quotient, and let $$\phi: (\widetilde{\calX}/G, \widetilde{\calD}/G) \to (\calX /G, \calD /G)$$ denote a good resolution. Then Lemma 4 from \cite{dan2} tells us that the invariant sections of $\omega_{\calX}(\calD)^{\otimes n}$ vanishing on the fixed point locus to order $\geq n|G|$ come from sections of the pluri-log canonical divisors of a desingularization, i.e. sections of $\omega_{\widetilde{\calX}}(\widetilde{\calD})^{\otimes n}$. Therefore, for $n$ sufficiently large, the quotient pair $(X^n_B, D_n)/G$ is openly of log general type. 
\end{proof}

This also proves the following theorem:

\begin{theorem}\label{thm38} Let $f: (X,D) \to B$ be a stable family with integral, openly canonical, and log canonical general fiber over a smooth projective variety $B$. Suppose that the variation of the family $f$ is maximal.  Let $G$ be a finite group such that $(X,D) \to B$ is $G$-equivariant.  Then there exists an integer $n > 0$ such that the quotient $(X^n_B/G, D_n/G)$ is openly of log general type. \end{theorem}

Furthermore, combining Proposition \ref{group} with Proposition \ref{sing} yields:

\begin{cor} \label{quotsing} Let $f:(X,D) \to B$ be an slc family with integral and log canonical general fiber over a smooth projective variety $B$. Then for $n$ large enough, the quotient pair $(X^n_B/G, D_n/G)$ also has log canonical singularities. \end{cor} 

This then gives an analogue to Proposition \ref{group} for pairs of log general type:

\begin{cor}\label{group2} Let $(X,D)$ be a pair of log general type. There exists a positive integer $n$ such that the pair $(X^n_B, D_n) / G$ is also a pair of log general type. \end{cor}
 
Thus we have completed the proof of the following Theorem \ref{thm3}. \\

\begin{theorem} \label{thm3} Let $(X,D) \to B$ be a stable family with integral and log canonical general fiber over a smooth projective variety $B$. Suppose that the variation of the family $f$ is maximal (see Definition \ref{max}). Let $G$ be a finite group such that $(X,D) \to B$ is $G$-equivariant. Then there exists an integer $n > 0$ such that the quotient of the pair by a finite group of automorphisms, $(X_B^n/G, D_n/G)$ is of log general type. \end{theorem}

\begin{proof} This follows from Theorem \ref{thm2} and Corollary \ref{group2}. \end{proof}

The next and final section shows how to reduce the proof of the Theorem \ref{fiber} to Theorem \ref{thm3}. Then, we show that Theorem \ref{fiber2} follows from Theorem \ref{fiber}.


\section{Proof of Theorems \ref{fiber} and \ref{fiber2} -- Reduction to case of max variation} \label{4}
 The final section of this paper is devoted to reducing the proofs of our two main theorems to the case of maximal variation. We will use the existence of a tautological family over a finite cover of our moduli space to show that, after a birational modification of the base, the pullback of a stable family with integral and log canonical general fiber has a morphism to the quotient of a family of maximal variation by a finite group. Then using the fact that our result holds for families of maximal variation, we will conclude that, after a modification of the base, a high fibered power of the pullback of a stable family with integral and log canonical general fiber has a morphism to a pair of log general type. 
 
 Finally, we show that if we add the assumption that the general fiber of our family is openly canonical and log canonical, we can avoid taking a modification of the base to prove Theorem \ref{fiber2}.
 
 \begin{remark}As we will be using the moduli space of stable pairs $\overline{M}_h$,  we remind the reader that we are in the situation of Assumption 1.16. \end{remark}

Unfortunately the moduli space $\overline{M}_h$ that we are working with does \emph{not} carry a universal family. The following lemma gives a \emph{tautological family}, which can be thought of as an approximation of a universal family. \begin{lemma}[{\cite[Corollary 5.19]{kp}}]\label{univ} There exists a tautological family $(\calT, \calD)$ over a finite cover $\Omega$ of the moduli space $\overline{M}_h$ of stable log pairs. That is, there exists a variety $\Omega$,  a finite surjective map $\phi: \Omega \to \overline{M}_h$ and a stable family $\calT \to \Omega$ such that $\phi(x) = \left[(\calT_x, \calD_x) \right]$. \end{lemma}

\begin{prop}\label{42} Let $f: (X,D) \to B$ be a stable family such that the general fiber is integral and has log canonical singularities. Then there exists a birational modification of the base $\widetilde{B} \to B$, and a morphism $(\widetilde{X},\widetilde{D}) \to \widetilde{B}$ to $(\calT_{\widetilde{\Sigma}}, \widetilde{\calD})/G$, the quotient of a family of maximal variation by a finite group $G$ \end{prop}

\begin{proof} Let $f: (X,D) \to B$ be a stable family such that the general fiber is integral and has log canonical singularities. In particular, we do \emph{not} assume that the variation of $f$ is maximal. There is a well defined canonical morphism $B \to \overline{M}_h$. Call the image of this morphism $\Sigma$. Over this $\Sigma$ lies the universal family $(\calT_{\Sigma}, \calD)$. Since $\overline{M}_h$ is a stack, the maps $(X,D) \to (\calT_{\Sigma}, \calD)$ and $B \to \Sigma$ factor through the coarse spaces: $\underline{\Sigma}$ and $(\underline{\calT}_{\Sigma}, \underline{\calD})$. The general fiber of $(\underline{\calT}_{\Sigma}, \underline{\calD}) \to \underline{\Sigma}$ is simply $(S, D_S) / K$ where $(S, D_S)$ is a pair of log general type and $K$ is the finite automorphism group. \\

Unfortunately there is no control on the singularities of $\Sigma$ -- if the singularities are not too mild, the fibered powers $(\calT^n_{\Sigma}, \calD_n)$ have no chance of having log canonical singularities.  To remedy this we take a resolution of singularities. Using Proposition \ref{univ}, we take a Galois cover followed by an equivariant resolution of singularities to obtain $\widetilde{\Sigma} \to \underline{\Sigma}$. Call the Galois group of this cover $H$. Then over $\widetilde{\Sigma}$, we have a tautological family $(\calT_{\widetilde{\Sigma}}, \widetilde{\calD})$. Here the general fiber is simply $(S, D_S)$, a pair of log general type.

Consider the quotient map $\widetilde{\Sigma} \to \widetilde{\Sigma}/H$. Taking the pullback of $(\underline{\calT}_{\Sigma}, \underline{\calD})$ through $\widetilde{\Sigma} / H$ yields $(\calT_{\widetilde{\Sigma} /H}, \widetilde{D}')$. Letting $G$ be the group $G = H \times K$, we can construct the following diagram: \begin{center} \begin{tikzcd}
(\calT_{\widetilde{\Sigma}}, \widetilde{D}) \ar{r}\ar{d} & (\calT_{\widetilde{\Sigma}}, \widetilde{\calD})/G \ar{r}{\nu}\ar{d} & (\calT_{\widetilde{\Sigma}/H},  \widetilde{\calD}') \ar{r}\ar{d} & (\underline{\calT}_{\Sigma}, \underline{\calD}) \ar{d}\ar{r} & \calT \ar{d} \\
\widetilde{\Sigma} \ar{r} & \widetilde{\Sigma} / H \ar{r} & \widetilde{\Sigma} / H \ar{r} & \underline{\Sigma} \ar{r} & \overline{M}_h \\
\end{tikzcd} 
\end{center}

We claim that the map $\nu: (\calT_{\widetilde{\Sigma}}, \widetilde{D})/G \to (\calT_{\widetilde{\Sigma}/H}, \widetilde{\calD}')$ is actually the normalization of $(\calT_{\widetilde{\Sigma}/H}, \widetilde{\calD}')$. First note that $(\calT_{\widetilde{\Sigma}}, \widetilde{\calD}) /G$ is normal, and that the morphism $\nu$ is finite since the morphism $(\calT_{\widetilde{\Sigma}}, \widetilde{\calD}) \to (\underline{\calT}_{\Sigma}, \underline{\calD})$ is. Therefore, to prove $\nu$ is the normalization of $ (\calT_{\widetilde{\Sigma}/H}, \widetilde{\calD}')$, it suffices to prove that $\nu$ is birational. To do so, consider the following diagram:

\begin{center} \begin{tikzcd}
(\calT_{\widetilde{\Sigma}}, \widetilde{\calD}) \ar{r}\ar{d} & (\calT_{\widetilde{\Sigma}}, \widetilde{\calD}) / H \ar{r}\ar{d} & \left((\calT_{\widetilde{\Sigma}}, \widetilde{\calD}) / H\right) / K = (\calT_{\widetilde{\Sigma}}, \widetilde{\calD}) / G \ar{d} \\
\widetilde{\Sigma} \ar{r} & \widetilde{\Sigma} /H \ar{r} & \widetilde{\Sigma} /H  \\
\end{tikzcd}
\end{center}

From this diagram it is clear that the general fiber of $(\calT_{\widetilde{\Sigma}}, \widetilde{\calD}) /G \to \widetilde{\Sigma} /H$ is precisely $(S, D_S) /K$ -- the quotient by $H$ identifies fibers and the quotient by $K$ removes the automorphisms. Since the map $\nu$ is an isomorphism over the generic fibers, $\nu$ is a birational map and thus is the normalization of  $(\calT_{\widetilde{\Sigma}/H}, \widetilde{\calD}')$.\\

The pair $(X,D)$ does \emph{not} map to $(\calT_{\widetilde{\Sigma}/H}, \widetilde{\calD}')$.  Instead, consider a modification of the base $\widetilde{B} \to B$ where $\widetilde{B} = B \times_{\Sigma} \widetilde{\Sigma}/H$.
Then the pullback $(\widetilde{X}, \widetilde{D})$ maps to $(\calT_{\widetilde{\Sigma}/H}, \widetilde{\calD}')$. Since the pair $(\widetilde{X}, \widetilde{D})$ is normal and $\nu$ is the normalization, we see that $(\widetilde{X}, \widetilde{D})$ also maps to $(\calT_{\widetilde{\Sigma}}, \widetilde{\calD}) /G$. Finally, because the family $(\calT_{\widetilde{\Sigma}}, \widetilde{\calD}) /G \to \widetilde{\Sigma} /H$ is the quotient of a family of maximal variation by a finite group, we have completed the proof of the proposition. \end{proof}

\begin{proof}[Proof of Theorem \ref{fiber}]
Let $(\calT_{\widetilde{\Sigma}}, \widetilde{\calD})$ denote the tautological family of maximal variation obtained in the proof of Proposition \ref{42}. Passing to $n$th fibered powers, Theorem \ref{thm3} guarantees that $(\calT^n_{\widetilde{\Sigma}}, \widetilde{\calD}_n)$ is of log general type. By Corollary \ref{group2}, $(\calT^n_{\widetilde{\Sigma}}, \widetilde{\calD}_n)/G$ is also of log general type for $n$ sufficiently large.  Thus the proof of Theorem \ref{fiber} follows from the above Proposition \ref{42}, as we have shown that after modifying the base, we obtain a morphism from a high fibered power of our family to a pair of log general type.
\end{proof}

Finally, we prove Theorem \ref{fiber2}, the fibered power theorem for pairs openly of log general type.

\begin{proof}[Proof of Theorem \ref{fiber2}]
This proof essentially follows from the proof of Proposition \ref{42}. Assuming that the general fiber is openly canonical and log canonical, Theorem \ref{thm38} shows that, for $n$ sufficiently large, the pair $(\calT^n_{\widetilde{\Sigma}}, \widetilde{\calD}_n)/G$ is openly of log general type. Since there is a birational morphism  $(\calT^n_{\widetilde{\Sigma}}, \widetilde{\calD}_n)/G \to (\underline{\calT}^n_{\Sigma}, \underline{\calD}_n)$, it follows that $(\underline{\calT}^n_{\Sigma}, \underline{\calD}_n)$ is also openly of log general type.  Therefore, we have constructed a morphism from a high fibered power of our family to a pair openly of log general type, and have thus completed the proof of the theorem. The upshot here, is that we do not have to modify the base of our starting family.
\end{proof}

\bibliographystyle{alpha}
\bibliography{fiberedpower}
\end{document}